\documentclass[10pt,reqno]{amsart}
\usepackage{fullpage}
\usepackage{amsfonts,amssymb,amsthm,amsbsy,latexsym,amscd,amsmath,euscript,enumitem,manfnt, marvosym,verbatim,calc,mathrsfs,tensor,textcomp,color,xcolor,tikz-cd,setspace,upgreek,bbm,bm}
\usepackage{amsgen,graphicx,dsfont,textcomp,pifont}
\usepackage[citecolor=blue,colorlinks=true]{hyperref}
\usepackage[all]{xy}
\newtheorem{defn}{Definition}[section]
\newtheorem{prop}[defn]{Proposition}

\newtheorem{lem}[defn]{Lemma}
\newtheorem{thm}[defn]{Theorem}
\newtheorem{cor}[defn]{Corollary}
\newtheorem{rem}[defn]{Remark}

\newcommand {\C}{{\mathds C}}

\newcommand {\R}{{\mathds R}}

\def\Ker{\operatorname{Ker}}

\def\dim{\operatorname{dim}}

\def\Im{\operatorname{Im}}

\def\Hom{\operatorname{Hom}}

\def\Rep{\operatorname{Rep}}
\def\Fun{\operatorname{Fun}}
\def\Vec{\operatorname{Vec}}
\def\sgn{\operatorname{sgn}}
\title{On Perverse sheaves of a Coxeter hyperplane arrangement of type $\mathcal{A}_n$}
\author{Umesh V Dubey}

\address{Harish-Chandra Research Institute, A CI of HBNI, Prayagraj, UP,  INDIA.}  
\email{umeshdubey@hri.res.in}

\author{Subham Sarkar}

\address{Kerala School of Mathematics,  INDIA.}  \email{subham.sarkar13@gmail.com, subham13@ksom.res.in}


\begin{document}

	\baselineskip=12pt
	\maketitle
	
	\begin{abstract}
		Kapranov and schechtman gave quiver description of perverse sheaves on real hyperplane arrangements. We used this description to relate the perverse sheaves on Coxeter hyperplane arrangements of type $\mathcal A_n$ for different values of $n$. As a consequence we prove that the simple perverse sheaves whose stalk on open cells are zero are induced from the perverse sheaves on lower dimension arrangements.
	\end{abstract}
	
	\section{Introduction}
	
	For an algebraic variety $X$ defined over the field of complex numbers $\C$, let $\mathscr{P}er(X,\mathcal{S})$ be the category of perverse sheaves with respect to a Whitney stratification $\mathcal{S}$. It is an abelian category obtained as a heart of a suitable $t$-structure defined by a perversity condition in the triangulated category $D^b_{const}(X,\mathcal{S})$, the bounded derived category of complexes of sheaves of $\C$ vector spaces whose cohomology sheaves are constructible with respect to the stratification $\mathcal{S}$ (\cite{BBD}). In certain special cases  the category $\mathscr{P}er(X,\mathcal{S})$ has a few combinatorial descriptions due to \cite{MV2} and \cite{KV}. 
	
	Let $X$ be a linear subspace of dimension  $n$ and $H_i$ be a hyperplane defined over the field of real numbers  in $X\cong \C^n$ for all $1\leq i\leq N$. Then $X$ has a natural Whiney stratification $\mathcal{S}$. The abelian category of perverse sheaves  $\mathscr{P}er(X,\mathcal{S})$ is equivalent to the category of representation of double quivers with a certain set of relations known as \emph{monotonicity, invertibility,} and \emph{ transitivity} \cite{KV}. In \cite{Bapat}, the category $\mathscr{P}er(X,\mathcal{S})$ is equivalent to the abelian category of modules over a noncommutative ring $R$ when $X$ is Coxeter hyperplane arrangements of type $\mathcal{A}_n$. Here we prove the following.
	
	Let $V^n$ be the linear subspace in $\R^n$ defined by $\{\sum_{ i=1}^{n+1} x_i=0\}\subset \mathds{R}^{n+1}$. A Coxeter  arrangement of hyperplanes in $V^n$ defined by the collection of hyperplanes $\mathcal{A}_n$ consisting of hyperplanes
	\[L^n_{i,j}:=\{(x_1,x_2,\dots,x_{n+1})\in V^n:x_i-x_j=0, \hspace{17pt} 1\leq i< j\leq n+1\}.\]
	The hyperplane arrangements of type $\mathcal{A}_n$ determine a natural Whitney stratification $\mathcal{S}^n$ on $V^n_{\C}$ ($\S$ \ref{ii}). For $1\leq i<j\leq n+1$, the linear map $$\rho_{L^{n+1}_{i,j}}:V^n\rightarrow V^{n+1}$$ defined by
	$$(x_1\dots,x_{i-1},x_i,x_{i+1},\dots,x_n,-(\sum_k x_k))\mapsto((x_1\dots,x_{j-1},x_i,x_{j+1},\dots,x_{n+1},-(\sum_{k\neq i} x_{k}+2x_i)) $$ is not a stratified algebraic morphism between two arrangements $(V^n_{\C},\mathcal{S}^n)$ and $(V^{n+1}_{\C},\mathcal{S}^{n+1})$. However, we have	
	\begin{thm}
		For any $1\leq i<j\leq n+1$,	there  is a fully-faithful exact functor
		
		\begin{equation}
		\Lambda_{L_{i,j}^{n+1}}:\mathscr{P}er(V_{\C}^n,\mathcal{S}^n)\rightarrow \mathscr{P}er(V_{\C}^{n+1},\mathcal{S}^{n+1}),
		\end{equation}
		which makes the following diagram commutative
		$$
		\begin{tikzcd}
		\mathscr{P}er(V_{\C}^n,\mathcal{S}^n)\arrow{r}{\mathcal{Q}^n}\arrow{d}{\Lambda_{L_{i,j}^{n+1}}}& \mathcal{J}_n\arrow{d}{\Phi_{L_{i,j}^{n}}}\\
		\mathscr{P}er(V_{\C}^{n+1},\mathcal{S}^{n+1})\arrow{r}{\mathcal{Q}^{n+1}}& \mathcal{J}_{n+1}.
		\end{tikzcd}
		$$ 
		
	\end{thm}	
	
	In \cite{RW}, the irreducible objects in 	$\mathscr{P}er(V_{\C}^n,\mathcal{S}^n)$ has been completely classified for $n=1$ and partially for $n=2$. In the same spirit, we obtain the following corollary.  	
\begin{cor}
		Let $\mathcal{F}$ be a constructible sheaves on the hyperplane arrangement $\mathcal{A}_{n+1}$. Then we have the following.
		\begin{enumerate}
			\item  If $\mathcal{F}\in \mathscr{P}er(V_{\C}^{n+1},\mathcal{S}^{n+1})$ is irreducible then  $\dim_{\C}(E_C(\mathcal{F}))=0$ or $1$, where $C$ be any open cell in $\mathcal{C}_{\mathcal{A}_{n+1}}$.   
			\item  Suppose $\mathcal{F}\in \mathscr{P}er(V_{\C}^{n+1},\mathcal{S}^{n+1})$ is irreducible and  $\dim_{\C}(E_C(\mathcal{F}))=0$ where $C$ be any open cell in  $\mathcal{C}_{\mathcal{A}_{n+1}}$. Then there exists $\mathcal{G}\in \mathscr{P}er(V_{\C}^{n},\mathcal{S}^{n})$ and $L_{i,j}^{n+1}\in \mathcal{H}_{\mathcal{A}_{n+1}}$ such that  \[\Lambda_{L_{i,j}^{n+1}}(\mathcal{G})=\mathcal{F}.\]
		\end{enumerate}
	\end{cor}	
	
	\section{Perverse sheaves of real hyperplane arrangements }
	
	In this section we will recall the quiver description of Kapranov and Schechtmann \cite{KV} which we will use in later sections.
	
	Let $X$ be an algebraic variety over $\C$ and $\mathcal{S}$ be a Whiney stratification on $X(\C)$. We denote by $\mathscr{P}er(X,\mathcal{S})$ the category of perverse sheaves with respect to $\mathcal{S}$ satisfying the following perversity conditions. Let $\mathcal{F}\in D^b_{const}(X,\mathcal{S})$.
	\begin{itemize}
		\item $\emph{CONDITION P}^{-}:$ For each $i$, the $i^{th}$ degree cohomology sheaf $
		\underline{H}^{i}(\mathcal{F})$ is supported on a closed subspace of codimension $\geq i$. 
		\item $\emph{CONDITION P}^{+}:$ The sheaf $\underline{	H}^{i}(l^{!}\mathcal{F})=\underline{\mathds{H}}^i_{Z}(\mathcal{F})$ is zero, for all $i<p$ where $l:Z\hookrightarrow X$ is a locally closed analytic submanifold of $X$ of codimension  $p$.  
	\end{itemize}
	
	\subsection{Sheaves on cellular space and linear algebra data}
	
	A topological space $X$ with a filtration $\upchi=\{X_d:\dim X_d=d\geq 1\}$ by closed subspace is said to be a cellular space if $X_d\setminus X_{d-1}$ is the disjoint union of cells of dimension $d$. Let $\mathcal{C}_{X,\upchi}$ be the set of all cells corresponding to the filtration $\upchi$ on $X$ which has the following partial order $\leq$.
	\begin{equation}\label{eq2}
	\text{For any } C \text { and } C'\in \mathcal{C}_{X,\upchi}, C'\leq C\text{ if and only if } C'\subset \bar{C}.
	\end{equation} 
	\begin{defn}	
		A sheaf $\mathcal{F}$ is said to be a cellular sheaf if for any $C\in\mathcal{C}_{X,\upchi}$, $j_{C}^*\mathcal{F}$ is a constant sheaf with stalk a finite-dimensional complex vector space where $j_{C}:C\hookrightarrow X$. 
	\end{defn}	
	Let us denote $Sh_{X,\upchi}$ to be the category of cellular sheaves on $X$ with respect to $\upchi$. For any $\mathcal{F}\in Sh_{X,\upchi}$, consider $\mathcal{F}_{C}:=H^{0}(C,j_{C}^*\mathcal{F})$, the stalk at $C\in \mathcal{C}_{X}$. For any $C$, $C'\in \mathcal{C}_{X}$ satisfying $C'\leq C$, define 
	\begin{equation}
	\gamma_{C'C}:\mathcal{F}_{C'}\rightarrow \mathcal{F}_{C}
	\end{equation}
	as follows. Take $y\in C'\subset V$ to be a sufficiently small contractible open neighborhood of $y$ and consider $x\in V\cap C'$. Let $U$ be a small contractible neighborhood of $x\in C$ in $X$. Since $U\subset V$, define $\gamma_{C'C}$ to be the restriction map 
	\[H^{0}(V,\mathcal{F})=\mathcal{F}_{C'}\rightarrow H^{0}(U,\mathcal{F})= \mathcal{F}_{C}.\]
	Let $\Rep(\mathcal{C}_{X}):=\Fun(\mathcal{C}_{X},\Vec_{\C})$ be the category of functors from  $\mathcal{C}_{X}$ to $\Vec_{\C}$, the category of finite dimensional vector spaces over $\C$.   
	\begin{prop}\cite[Proposition 1.8]{KV} 
		\begin{enumerate}
			\item The following covariant functor is an equivalence of categories  
			\begin{align*}
			Sh_{X,\upchi}\longrightarrow& \Rep(\mathcal{C}_{X})\\
			\mathcal{F}\mapsto& (\mathcal{F}_{C},\gamma_{C'C}). 
			\end{align*}
			\item Let $D^{b}_{\upchi}(Sh_{X})$ be the bounded derived category of sheaves whose cohomology sheaves lies in $Sh_{X,\upchi}$. The above equivalence of categories induces the  following equivalence of the respective derived categories
			\begin{equation}\label{ira1}
			D^{b}(\Rep(\mathcal{C}_{X}))\longrightarrow D^{b}_{\upchi}(Sh_{X}).
			\end{equation}
		\end{enumerate}
	\end{prop}
	In other words any object in $\mathcal{F}^{\bullet}\in D^{b}_{\upchi}(Sh_{X})$ determines a complex of vector spaces $(\mathcal{F}_{C}^{\bullet})_{C\in \mathcal{C}_{X}}$ with generalisation map of complex $\gamma_{C'C}^{\bullet}:\mathcal{F}_{C}^{\bullet}\longrightarrow \mathcal{F}_{C'}^{\bullet}$, where $C'\leq C\in \mathcal{C}_{X}$. A category of perverse sheaves can be defined as an abelian subcategory of $D^{b}_{\upchi}(Sh_{X})$. Therefore if $\mathcal{F}^{\bullet}$ is perverse, it determines a set of \emph{linear algebra data} (a complex of vector spaces with generalization maps).

	\subsection{Real hyperplane arrangements}\label{ii}
	
	Let $V$ be a vector space defined over $\R$ of dimension $n\geq1$ and  $\mathcal{H}$ be an arrangement of linear hyperplanes in $V$. Here  we only consider \emph{essential} hyperplane arrangements. In otherwords,  $\underset{H\in\mathcal{H}}{\cap}H=\{0\}$. Each hyperplane $H\in\mathcal{H}$ defined by a linear equation $f_H:V\rightarrow\R$. We choose once and for all a set of equations $\{f_H:H\in\mathcal{H}\}$.

	Each $x\in V$ determines a sign vector $(v_H(x))_{H\in\mathcal{H}}\in\mathcal{H}^{\{+,-,0\}}$, where $v_H(x):=\sgn(f_H(x))$ and $\sgn:\R\rightarrow \{+,-,0\}$ be the sign function. Any two distinct points $x$ and $y\in V$ are said to be lie in the same face if and only if $v_H(x)=v_H(y)$, for all $H\in\mathcal{H}$. Hence we identify the set of faces with the corresponding set of sign vectors associated with $\mathcal{H}$. The faces are locally closed  convex subsets of $V$ and it induces cellular decomposition of $V$. In particular, faces would be called cells depending on the context. Hence the hyperplane arrangement induces a natural filtration $\upchi$ on $V$. 
	
	Let $\mathcal{C}_{\mathcal{H},\upchi}$ be set of cells corresponding to the arrangement of hyperplanes $\mathcal{H}$ in $V$. The following two partial ordering in  $\mathcal{C}_{\mathcal{H},\upchi}$ are equivalent. 
	\begin{enumerate}
		\item For any $C$ and $C'\in \mathcal{C}_{\mathcal{H},\upchi}$, $C'\leq C$ if and only if  $C'\subset \bar{C}$.
		\item The following partial order on $\{+,-,0\}$ $$0\leq +,\hspace{5pt} 0\leq -$$
		gives a partial order on $\mathcal{C}_{\mathcal{H},\upchi}$ as follows. For any $C$ and $C'\in \mathcal{C}_{\mathcal{H},\upchi}$, $$C'\leq C \hspace{19pt}\text{ if and only if } \hspace{10pt} v_H(C')\leq  v_H(C)$$ for any hyperplane $H\in \mathcal{H}$.
	\end{enumerate}
	On the other hand there exist another natural filtration $\bar{\upchi}$ on $V_{\C}$ associated to the $\mathcal{H}$ where  
	$$X_d:=\bigcup_{\dim(C)\leq d}\R^n\oplus iC.$$
	It induces a cell decomposition $\mathcal{C}_{\mathcal{H},\bar{\upchi}}$ on $V_{\C}$ with the above partial order. There is a bijection between $\mathcal{C}_{\mathcal{H},\upchi}$ and $\mathcal{C}_{\mathcal{H},\bar{\upchi}}$ which preserve the partial orders. Hence we identify   $\mathcal{C}_{\mathcal{H},\upchi}$ with $\mathcal{C}_{\mathcal{H},\bar{\upchi}}$.
	
	\subsection{Real hyperplane arrangement and linear algebra data}
	
	An arrangement of hyperplanes $\mathcal{H}$ defines a natural Whitney  stratification on $V_{\C}\cong{\C}^{n}$ as follows. The set of flats is defined by $\mathcal{L}_{\mathcal{H}}:=\{\underset{H\in\mathcal{I}}{\cap}H:\mathcal{I}\subset \mathcal{H}\}.$ For any flats  $L\in \mathcal{L}$, define associated Whitney stratum 
	$L^{\textdegree}_{\C}:=L_{\C}\setminus(\cup_{H\nsupseteq L}H_{\C})$. The set $\mathcal{S}_0:=\{L^{\textdegree}_{\C}:L\in \mathcal{L}_{\mathcal{H}}\}$ gives a Whitney stratification on ${\C}^{n}$. Let us consider  $\mathscr{P}er({\C}^n,\mathcal{S}_0) $ to be the category of perverse sheaves on ${\C}^n$ with respect to $\mathcal{S}_0$  and it is an abelian category of the  triangulated category  $D^b({\C}^n,\mathcal{S}_0)$.
	
	For $\mathcal{F}\in D^{b}_{\mathcal{S}_0}(\C^n)$, consider $\mathcal{R}_{\mathcal{F}}^{\bullet}:=R\Gamma_{\R^n}(\mathcal{F}^{\bullet})[n]\in D^{b}_{\mathcal{C}_{\mathcal{H}}}(\R^n)$ which is a complex of cellular sheaves. Hence by \eqref{ira1}, $\mathcal{R}_{\mathcal{F}}^{\bullet}$  can be identified with  linear algebra data of complexes of vector spaces as follows. Let
	\[E_{C}^{\bullet}(\mathcal{F}):=R\Gamma(C,\mathcal{R}_{\mathcal{F}}^{\bullet})\]
	and 
	\[\gamma_{C'C}:E^{\bullet}_{C'}(\mathcal{F})\rightarrow E^{\bullet}_{C}(\mathcal{F}).\]
	For $\mathcal{F}\in \mathscr{P}er(V_{\C}^n,\mathcal{S}^n)\subset D^{b}_{\mathcal{S}_n}(\C^n)$, $E_{C}^{\bullet}(\mathcal{F})$ is quasi-isomorphic to a complex of vector spaces with only non-zero entry $E_{C}(\mathcal{F})$ at degree zero. In particular, for each $C\in \mathcal{C}_{\mathcal{H},\upchi}$, functor $\mathcal{F}\mapsto E_{C}(\mathcal{F})$ is an exact functor. Hence one obtains  $\gamma_{C'C}:E_{C'}(\mathcal{F})\rightarrow E_{C}(\mathcal{F})$
	as a linear map of vector spaces.

	
	Let $j_{C}:\R^n +i C\hookrightarrow \C^n$, where
	$C\in \mathcal{C}_{\mathcal{H},\upchi}$ of dimension $d$,  define 
	$$\mathcal{E}^{\bullet}_{C}(\mathcal{F}):=j_{C\star}j_{C}^{!}(\mathcal{F})[d]=R\Gamma_{\R^n+i C}(\mathcal{F})[d].$$
	
	%
	
	Recall the associated complex of homological Postnikov system, [Sec 4.A, \cite{KV}], of $\mathcal{F}$ which realises  $\mathcal{F}$ as a total object 
	\begin{equation}\label{3.23}
	\mathcal{E}^{\bullet}(\mathcal{F}):= \bigoplus_{\dim(C)=0}\mathcal{E}^{\bullet}_{C}(\mathcal{F})\rightarrow \bigoplus_{\dim(C)=1}\mathcal{E}^{\bullet}_{C}(\mathcal{F})\rightarrow\dots\rightarrow \mathcal{E}^{\bullet}_{0}(\mathcal{F}).
	\end{equation}
	Now applying the functor $R\Gamma(\C^n,\text{  })$ to the equation \eqref{3.23}, we get the associated total complex of the corresponding Postnikov system in the derived categories of vector spaces, known as \emph{Cousin complex} of $\mathcal{F}$.  
	\begin{equation}
	\bigoplus_{\dim(C)=0}E^{\bullet}_{C}(\mathcal{F})\otimes or(C)\overset{\tilde{\delta}}{\rightarrow} \bigoplus_{\dim(C)=1}E^{\bullet}_{C}(\mathcal{F})\otimes or(C)\overset{\tilde{\delta}}{\rightarrow}\dots\rightarrow E^{\bullet}_{0}(\mathcal{F})\otimes or(C).
	\end{equation}
	For $C'\leq C$ such that $\dim (C')=\dim(C)-1$, one obtains
	\begin{equation}
	\delta_{CC'}:E_{C}\rightarrow E_{C'}.
	\end{equation}
	%
	%
	This map extends to give $ \delta_{CC'}$ for any $C'\leq C$. We can also use the Verdier duality to redefine $\delta_{CC'} := \gamma^{\mathcal F^{\star}}_{C'C}$ where $\mathcal F^{\star} := \underline{R \Hom} (\mathcal F, or_X)$ denotes the Verdier dual of perverse sheaf $\mathcal F$, see [Prop 4.6, \cite{KV}]. This defines a perfect duality in the category of perfect complexes. from now onwards we use the notation $\mathcal{C}_{\mathcal{H}}$ instead of $\mathcal{C}_{\mathcal{H},\upchi}$

	\begin{defn}
		A double representation of $\mathcal{C}_{\mathcal{H}}$ is given by a datum $(E_C,\gamma_{C'C},\delta_{CC'})$, where $E_C$ is a $\C$-vector space for every $C\in \mathcal{C}_{\mathcal{H}}$ and for $C'\leq C$ linear maps
		\[\gamma_{C'C}:E_{C'}\rightarrow E_{C} \text{  and  } \delta_{CC'}:E_{C}\rightarrow E_{C'}
		\]
		Such that 
		\begin{itemize}
			\item $\gamma_{CC}=\delta_{CC}=Id_C$
			\item $\gamma_{CC''}=\gamma_{C'C''}\gamma_{CC'}$ and $\delta_{CC''}=\delta_{C'C''}\delta_{CC'}.$ where $C''\leq C'\leq C$.
			
		\end{itemize}
	\end{defn}
	\begin{defn}
		\begin{itemize}
			\item Let $C_1$ and $C_2$ be two distinct faces of dimensnion $d$ such that the linear span of $C_1$ and $C_2$ are equal. Then $C_1$ and $C_2$ are said to be opposed each other by a face $D$ of dimension $d-1$ if $D\leq C_1$, $D\leq C_2$; and for any hyperplane $H\in \mathcal{H}$, $\sigma(C)_{H}=0$ and $\sigma(A)_{H}=-\sigma(B)_{H}\neq0.$ 
			\item Any triple of faces $(A,B,C)$ is said to be colinear if any line joining a point in $A$ and $C$ lies in $B$.
		\end{itemize}
	\end{defn}
	Suppose $\Rep_2(\mathcal{C}_{\mathcal{H}})$ be the category of double representation of $\mathcal{C}_{\mathcal{H}}$. It is an abelian category. For any cells $A$, $B$ and $C\in\mathcal{C}_{\mathcal{H}}$ satisfying $C\leq A$ and $C\leq B$ we define $\phi_{AB}:=\gamma_{CB}\delta_{AC}$. Consider the full subcategory $\mathcal{J}_{\mathcal{H}}$ consisting of double representations which satisfy  the relations --\emph{monotonicity}, \emph{transitivity} and \emph{invertibility} which are defined as follows.
	\begin{itemize}
		\item \emph{Monotonicity}: For any $C\leq C'$, one has $\gamma_{C'C}\delta_{CC'}=Id_C$. 
		\item \emph{Transitivity}: For any colinear cells $C_1$, $D_1$ and $C_2\in \mathcal{C}_{\mathcal{H}}$, one has 
		$$\phi_{C_1C_2}=\phi_{D_1C_2}\phi_{C_1D_1}$$.
		\item \emph{Invertibility}: If $C_1$ and $C_2$  oppose  each other by $D_1$, then $\phi_{C_1C_2}$ is an isomorphism.
	\end{itemize} 

	We can also define a perfect duality on the category $Rep_2(\mathcal{C}_{\mathcal{H}})$.
	Define the dual of a double representation as $ (V_C, \gamma_{C'C}, \delta_{CC'})^{\star} := (V_C^*, \delta_{CC'}^*, \gamma_{C'C}^*)$, where $V_C^*$ is the vector space dual.  Now we can check that the duality functor preserves all three relations and hence it will restrict to the subcategory $\mathcal{J}_{\mathcal{H}}$ to give the perfect duality.
	\begin{thm}\cite[Theorem 8.1]{KV}\label{tthm}
		Let $\mathcal{H}$ be any real hyperplane arrangement on $\mathds{R}^n$ and $\mathcal{S}$ be a natural Whitney stratification associated to $\mathcal{H}$. There is an equivalence of categories
		\[\mathcal{Q}:\mathscr{P}er(\C^n,\mathcal{S}_0)\longrightarrow\mathcal{J}_{\mathcal{H}},\]
		\[\mathcal{F}\mapsto \left(E_{C}(\mathcal{F}, \gamma_{C'C}(\mathcal{F}),\delta_{CC'}(\mathcal{F}))\right)_{C'\leq C\in\mathcal{C}_{\mathcal{H}}}.\]
		where $\mathcal{J}_{\mathcal{H}}$ is the full subcategory of $\Rep_2(\mathcal{C}_{\mathcal{H}})$. Moreover, the functor $\mathcal Q$ commutes with the duality i.e. $\mathcal Q(\mathcal F^\star) = \mathcal Q (\mathcal F)^\star$.  	
	\end{thm}
	
	\section{Coxeter type arrangements}
	
	Coxeter hyperplane arrangement of type $\mathcal{A}_n$ is defined in \cite[$\S3$]{PR}. It is also known as a braid arrangement of type $\mathcal{A}_n$. 
	Let $V^n$ be a linear subspace of dimension $n$ over the field of real numbers, defined by $$V^n:=\{\sum_{ i=1}^{n+1} x_i=0\}\subset \mathds{R}^{n+1}.$$   
	Consider $V_{\C}^n$ to be its complexification. Here we consider the Coxeter arrangement of hyperplanes of type $\mathcal{A}_{n}$ on $V^n$.

	\begin{defn}
		A Coxeter arrangement of type $\mathcal{A}_{n}$ is  defined as a collection of hyperplanes $\mathcal{H}_{\mathcal{A}_n}$consisting the following set of equation in $V^n$
		\[L^n_{i,j}:=\{(x_1,x_2,\dots,x_{n+1})\in V^n:x_i-x_j=0,   \hspace{17pt} 1\leq i< j\leq n+1\}.\]
	\end{defn}
	\subsection{Embedding of triangulations }	
	A triangulation of a topological space $X$ is a cellular (CW-complex) structure on $X$. An hyperplane arrangement on a linear space $X$ gives a triangulation of $X$. Here we compare two distinct triangulations of a linear space  $X$ of dimension $n$, obtained by two different hyperplane arrangements $\mathcal{A}_{n}$ and the induced hyperplane arrangement $\mathcal{A}_{n+1}$ on a subspace  $L^{n+1}_{i,j}$. Since the equation of hyperplane in the above two cases are symmetric in $i$ and $j$, it is enough to consider the special case $i=1$ and $j=n+1$.  	
	Consider $$\rho_{L^{n+1}_{1,n+1}}:V^n\rightarrow V^{n+1}$$ defined by 
	$$(x_1\dots,x_n,-(\sum_i x_i))\mapsto((x_1\dots,x_{n-1},x_1,-(\sum_{i\neq1} x_{i}+2x_1)). $$
	We claim that the above map induces  an order preserving embedding 
	\[\iota_{L_{1,n+1}^{n+1}}:\mathcal{C}_{\mathcal{A}_{n}}\longrightarrow \mathcal{C}_{\mathcal{A}_{n+1}}.\]
	
	The hyperplane arrangement of type $\mathcal{A}_{n}$ can be identified with an arrangement $\mathcal{H}_{\R^n}$ on $\R^n$ as follows. 
	For each $n\geq1$, consider the natural isomorphism between real vector spaces of dimension $n$, $\phi_n:{\R}^n\longrightarrow V^n$, defined by $(y_1,y_2,\dots,y_n)\mapsto (x_1,x_2,\dots,x_{n+1})$, where $x_i=y_i$ for all $1\leq i\leq n$ and $x_{n+1}=-\sum_{j=1}^{j=n}y_j$.
	
	For all $H\in\mathcal{H}_n$, $\phi^*(H)$, the pullback hyperplanes are  the following
	\begin{align*}
	A_{i,j}^n:=&Y_i-Y_j \text { for }  1\leq i<j\leq n,\\
	B_i^n:=& 2Y_i+\sum_{j\neq i} Y_j   \text { for }  1\leq i\leq n.
	\end{align*}
	Consider the arrangement of hyperplane $\mathcal{H}_{\R^n}:=\{A_{i,j}^n, B_i^n:1\leq i<j\leq n\}$ of $\R^n$. 
	The above isomorphism induces a bijection between  $\mathcal{H}_{\R^n}$  and $\mathcal{A}_n$. Let us denote $W_n:=\R^{n+1}\cap A_{1,n+1}^{n+1}$ to be the real vector space of dimension $n$ with the following hyperplane  arrangement $\mathcal{H}_{W_n}$:
	\begin{align*}
	\tilde{A}_{i,j}^{n+1}&:=A_{i,j}^{n+1}\cap A_{1,n+1}^{n+1}=&Y_i-Y_j &\text { for }  1< i<j\leq n,\\
	\tilde{A}_{1,j}^{n+1}&:=A_{1,j}^{n+1}\cap A_{1,n+1}^{n+1}=&Y_1-Y_j &\text { for }  1<j\leq n,\\
	\tilde{A}_{i,n+1}^{n+1}&=A_{i,n+1}^{n+1}\cap A_{1,n+1}^{n+1}=&Y_i-Y_{n+1}=Y_i-Y_1 &\text { for }  1<i\leq n,\\
	\tilde{B}_i^{n+1}&:=B_i^{n+1}\cap A_{1,n+1}^{n+1} =& 2Y_1+2Y_i+\sum_{j\neq i,1} Y_j   &\text { for }  1< i\leq n,\\
	\tilde{B}_1^{n+1}&:=B_1^{n+1}\cap A_{1,n+1}^{n+1} =& 3Y_1+\sum_{j\neq i,1} Y_j.   &
	\end{align*}
	Moreover, we  claim  that the natural isomorphism $$\Psi_n:\R^n\longrightarrow W_n,$$
	$$(y_1,\dots,y_n)\mapsto (y_1,\dots,y_n,y_1)$$
	induces an order-preserving bijection $\iota_n:\mathcal{C}_{\R^{n}}\longrightarrow \mathcal{C}_{W_{n}}$. In particular, it is enough to prove that two different arrangements $\mathcal{H}_{\R^n}$ and the pullback of hyperplanes in $\mathcal{H}_{W_n}$ by $\Psi_n$ in $\R^n$, induces an order-preserving bijection between poset of cells in $\R^n$ with respect to $\mathcal{H}_{\R^n}$ and $\mathcal{H}_{W_n}$. For any $P\in \R^n$, define 
	\[E_i(P):=2z_i+\sum_{j\neq 1,i} z_j.\]

	\begin{thm}\label{imp}
		There is an order preserving bijection $\iota_n:\mathcal{C}_{\R^{n}}\longrightarrow \mathcal{C}_{W_{n}}$ which induces an order preserving embedding 
		\[\iota_{L_{1,n+1}^{n+1}}:\mathcal{C}_{\mathcal{A}_{n}}\longrightarrow \mathcal{C}_{\mathcal{A}_{n+1}}.\]
	\end{thm}
	\begin{proof}
		Let us recall that we have an isomorphism $\phi_n:\R^n\longrightarrow V^n$ and consider the hyperplane arrangements $\mathcal{H}_{\R^n}$, obtained by pullback of hyperplanes $L_{i,j}^n$, where $1\leq i<j\leq n$. The hyperplanes $\mathcal{A}_{n+1}$ induce  an arrangement of hyperplanes on $L_{1,n+1}^{n+1}$. We also have an isomorphism $\psi_n:\R^n\longrightarrow L_{1,n+1}^{n+1}$ which pullback the induced hyperplane arrangements and we obtain an arrangement $\mathcal{H}_{W_n}$. Hence we obtain two distinct arrangements $\mathcal{H}_{\R^n}$ and pullback of hyperplanes in $\mathcal{H}_{W_n}$ in $\R^n$. Let $\Lambda$ and $\Gamma$ be the sign vectors associated to $\mathcal{H}_{\R^n}$ and $\mathcal{H}_{W_n}$ respectively, defined as follows.
		\begin{align*}
		\Lambda:\R^n\longrightarrow &\mathcal{H}_{\R^n}^{\{+,-,0\}}\\
		x\mapsto&\left((v_{A_{i,j}^n}(x),v_{B^n_i}(x)\right)_{1\leq i<j\leq n}
		\end{align*}
		and 
		\begin{align*}
		\Gamma:\R^n\longrightarrow &\mathcal{H}_{W_n}^{\{+,-,0\}}\\
		x\mapsto&\left((v_{\tilde{A}_{i,j}^{n+1}}(x),v_{\tilde{B}^{n+1}_i}(x),v_{\tilde{A}_{i,n+1}^{n+1}}(x)\right)_{1\leq i<j\leq n}.
		\end{align*}
		The sign vectors associated to any hyperplane arrangement determine a cell decomposition on $\R^n$ uniquely. Hence, in order to show $\iota_n$ is an order-preserving bijection one needs to show
		
		\begin{itemize}
			\item For any $C\in \mathcal{C}_{\mathcal{H}_{\R^n}}$ and $v(C)\in \mathcal{H}_{n}^{\{+,-,0\}}$ there exist a $C_W\in \mathcal{C}_{W_{n}}$ such that $v(C_W)=(v(C),e_1,\dots,e_{n+1})$, where $e_i=v_{\tilde{A}_{i,n+1}^{n+1}}(C_W)\in\{+,-,0\}$.
			\item The above correspondence is an order preserving bijection.
		\end{itemize}
		
		For any cell $C\in \mathcal{C}_{\R^{n}}$, a  point in $C$ determines an  unique sign vector in $\mathcal{H}_{\R^n}^{\{+,-,0\}}$ and vice-versa. Hence 
		it is enough to prove that for any $P\in C\subset \R^n$, there exists a point $Q\in \R^n$ such that
		
		$$v(Q)=(v(P),e_1,\dots,e_{n}),\hspace{3pt} \text{where  }  e_i=v_{\tilde{A}_{i,n+1}^{n+1}}(Q)\in\{+,-,0\}.$$
		By Lemma (\ref{ita1}) and Lemma (\ref{iii}), there is an bijection between the collection of open cells corresponding to  $\mathcal{H}_{\R^n}$ and $\mathcal{H}_{W_n}$ in $\R^n$.

	\end{proof}

	\begin{lem}\label{ita1}
		There is an order-preserving bijection between the poset of cells of codimension zero associated to the hyperplane arrangements of $\mathcal{H}_{\R^n}$ and $\mathcal{H}_{W_n}$ in $\R^n$.
		
	\end{lem}
	\begin{proof}
		Let $C\in \mathcal{C}_{\mathcal{H}_{\R^n}}$ be an open cell. An open cell in $\R^n$,  corresponding to $\mathcal{H}_{\R^n}$  determines a sign vector with only non-zero entries. In particular, for any $P\in C$, all of the entries of $v_{A_{i,j}^{n}}(P)$ are either $+ve$ or $-ve$ sign,  for all $1\leq i\neq j\leq n$. Hence for any point $P=(z_1,\dots,z_n)\in C$, one has the following  presentation

		\begin{equation}\label{ra11}
		z_{\lambda_1}<z_{\lambda_2}<\dots<z_{\lambda_i}<z_{\lambda_{i+1}}<\dots<z_{\lambda_{n}},
		\end{equation}
		The sign vector $v(C)$ must determines one of the following relations
		\begin{enumerate}\label{ii1}
			\item $v_{B_i^{n}}(P)>0$, for all $i$
			\item  $v_{B_i^{n}}(P)<0$, for all $i$ or
			\item  $v_{B_{i_j}^n}(P)>0$ and
			$v_{B_{m_j}^n}(P)<0$, where $1\leq i_j$, $m_j\leq n$. 
		\end{enumerate} 
		Then it is enough to prove that there exist a point $Q=(z_1',\dots z_n')$ satisfying \eqref{ra11} and also satisfy one of the following relations
		\begin{enumerate}
			\item[I.] For all $1\leq i\leq n$, if $v_{B_i^{n}}(P)>0$,  then $v_{B_i^{n+1}}(Q)>0$,
			\item[II.]  For all $1\leq i\leq n$, if $v_{B_i^{n}}(P)<0$,  then $v_{B_i^{n+1}}(Q)<0$ or
			\item[III.]  If	$v_{B_{i_j}^{n}}(P)>0$ and
			$v_{B_{m_j}^n}(P)<0$, then $v_{\tilde{B}_{i_j}^{n+1}}(Q)>0$ and
			$v_{\tilde{B}_{m_j}^{n+1}}(Q)<0$  where $1\leq i_j$, $m_j\leq n$. 
		\end{enumerate}
		Let us assume that $z_1\geq0$. 
		\begin{itemize}
			\item Case [I]:
			
			Suppose we have $v_{B_i^{n}}(P)>0$, for all $i$. Then Choose $Q=P$, one has $v_{B_i^{n+1}}(Q)>0$.
			\item Case [III]:
			Let  $v_{B_{i_j}^{n}}(P)>0$ and
			$v_{B_{m_j}^n}(P)<0$, where $1\leq i_j$, $m_j\leq n$.
			
			Hence,
			\[z_1+E_{i_j}(P)>0,\hspace{5pt} z_1+E_{m_j}(P)<0.\]
			Consider
			\[M_P:=Min_{m_j}\{-E_{m_j}(P)/2\}>0.\]
			If $z_1< M_P$, consider $Q=P$ and $v_{\tilde{B}_{i_j}^{n+1}}(Q)>0 \hspace{3pt} \text{and }
			v_{\tilde{B}_{m_j}^{n+1}}(Q)<0$ holds. 
			Suppose $z_1> M_P>0$. Then the entries of the point $P$ must have one of the following presentations  
			\begin{equation}\label{ra1}
			0<M_P< z_{\lambda_1}=z_1<z_{\lambda_2}<\dots<z_{\lambda_i}<z_{\lambda_{i+1}}<\dots<z_{\lambda_{n}},
			\end{equation}
			or, for some $1\leq i\leq k-1$ and $0\leq j\leq i-2$,
			\begin{equation}\label{ta11}
			z_{\lambda_1}<\dots<z_{\lambda_{k-i-1}}<0<z_{\lambda_{k-i}}<\dots< z_{\lambda_{k-i+j}}<M_P< z_{\lambda_{k-i+j+1}}<\dots <z_{\lambda_{k-1}}=z_1<\dots<z_{\lambda_{n}}.
			\end{equation}
			If the entries of the point $P$ satisfy the presentation given in \eqref{ra1}, consider $z_1'$ any real number such that $0<z_1'<M_P<  z_{\lambda_1}=z_1$ and $z_t'=z_t$ for all $1<t\leq n$. Since $E_{m_j}(Q)=E_{m_j}(P)$ and $z_1'<M_P\leq -E_{m_j}/2$ implies $2z_1'+E_{m_j}(P)<0$. Thus 
			$$B_{m_j}^{n+1}(Q)=2z_1'+E_{m_j}(P)<0.$$

			Suppose the entries of the point $P$ is given by \eqref{ta11}. 
			Let \[m:=z_{\lambda_{k-i+j+1}}-M_P.\] 
			Choose $\epsilon>0$ such that 
			\[z_{\lambda_{k-i+j}}<M_p-\epsilon<M_P.\]
			Pick $$z_{\lambda_{k-i+j+1}}'=z_{\lambda_{k-i+j+1}}-m-{\epsilon}=M_P-\epsilon,$$
			$$z_{\lambda_{k-i+j+2}}'=z_{\lambda_{k-i+j+1}}-m-{\epsilon}/2=M_P-\epsilon/2,$$
			and finally
			$$z_{\lambda_{k-2}}'=z_{\lambda_{k-i+j+1}}-m-\frac{\epsilon}{j-i}=M_P-\frac{\epsilon}{j-i}$$
			and $z_{\lambda_t}'=z_{\lambda_t}$, for all  $t\leq k-i+j$ and  $t\geq\lambda_{k}$. In particular one has 
			\begin{equation}\label{tataa}
			z_t      \leq z_{t}'.                                                                \end{equation}

			There always exist a real number $z_1'$ such that $M_P-\frac{\epsilon}{j-i}<z_1'<M_P$.
			Since, $z_1'<M_P\leq -E_{m_j}/2$ implies $2z_1'+E_{m_j}(P)<0$.  Hence
			\begin{align*}
			2z_1'+E_{m_j}(P')=&2z_1'+2z_{m_j}'+\sum_{l\neq m_j,1}z_l'
			\leq&2z_1'+2z_{m_j}+\sum_{l\neq m_j,1}z_l\hspace{5pt} (\text{ using } \eqref{tataa}  ) \\
			& &=2z_1'+E_{m_j}(P)<0.
			\end{align*}
			
			Hence $\tilde{B}_{m_j}^{n+1}(Q)<0$ and $\tilde{B}_{i_j}^{n+1}(Q)>0$ for $1\leq i_j$, $m_j\leq n$.
			
			The case $z_1=M_p$ is similar to the above case.

			\item Case [II]:
			
			Let $v_{B_i^n}(P)<0$ for all $1\leq i\leq n$. If $z_1<M_P$, consider $Q=P$ and 
			$v_{B_{i}^{n+1}}(Q)<0$ holds. Otherwise we can proceed as follows. Since $v_{\tilde{B}_{i}^{n+1}}(Q)<0$ for all $1\leq i\leq n$, the entries of the the point $P$ can only have the presentation in \eqref{ta} (Since $z_i$ are all positive). Now the proof follows by doing steps as in the Case [III].

		\end{itemize}
		
		The case $z_1<0$ is  similar to the case $z_1\geq0$, so we skip the proof here.
		Hence if $C$ be any open cell in $\mathcal{C}_{\R^n}$ satisfy any of the condition, there exist an cell $C'\in \mathcal{C}_{W_n}$ satisfy the similar condition.	
	\end{proof}
	
	\begin{lem}\label{iii}
		There is an order-preserving bijection between the poset of cells of codimension $\geq1$ associated to the hyperplane arrangements of $\mathcal{H}_{\R^n}$ and $\mathcal{H}_{W_n}$ in $\R^n$.
	\end{lem}
	\begin{proof}
		Let $C\in\mathcal{C}_{\mathds{R}^n}$ be a cell of codimension $\geq1$ such that at least one of entries of the sign vectors of $v(C)$ is zero. Equivalently, for every point $P\in C$, such that either $A_{i,j}^n(P)=0$ or $B^{n}_i(P)=0$, for some $1\leq i<j\leq n$. 
		Let us consider the following cases.
		\begin{itemize}
			\item Case [I]:
			None of $A_{i,j}^n$ is vanishing at $P$. In particular, the entries of the point $P\in C$ has the following representation
			\begin{equation}\label{ra13}
			z_{\lambda_1}< z_{\lambda_2}< \dots< z_{\lambda_i}< z_{\lambda_{i+1}}<\dots< z_{\lambda_{n}}.
			\end{equation} 
			Then $B^{n}_i(P)=0$ for exactly one $i$ satisfying $1\leq i\leq n$. Indeed, for any $P=(z_1,z_2,\dots,z_n)$ and $a\leq i_1$, $i_2\leq n$, $$v_{B_{i_1}^n}(P)=v_{B_{i_2}^n}(P)=0,$$ 
			implies $z_{i_1}=z_{i_2}$. 
			Hence $A_{i_1,i_2}^n(P)=0$, a contradiction.
			The entries of the sign vector $v(C)$ corresponding to $B_i^n$ must be of the form 
			\begin{equation}\label{equt}v_{B_{i_j}^{n}}(P)>0, \hspace{5pt} v_{B_{l}^{n}}(P)=0\hspace{3pt} \text{ and }
			v_{B_{m_j}^n}(P)<0,\end{equation}
			where $1\leq i_j$, $l$, $m_j\leq n$ for all $j$. 
			Our claim there is a point $Q\in\R^n$ satisfying the presentation in \eqref{ra13} such that \begin{equation}\label{tate}v_{\tilde{B}_{i_j}^{n+1}}(Q)>0, \hspace{5pt} v_{\tilde{B}_{l}^{n+1}}(Q)=0\hspace{3pt} \text{ and }
			v_{\tilde{B}_{m_j}^{n+1}}(Q)<0,\end{equation}
			
			From \eqref{equt}, one obtains for $1\leq i_j$, $l$, $m_j\leq n$
			\begin{equation}
			z_1+E_{i_j}(P)>0, \hspace{5pt} z_1+E_{m_j}(P)<0 \hspace{3pt}\text{ and }z_1=-E_{l}(P).
			\end{equation}
			In particular one has 
			\begin{equation}\label{taa1}
			z_1=-E_l(P)<-E_{m_j}(P) \hspace{3pt} \text{and  } -E_l(P)<z_{\lambda_n}\hspace{3pt} \text{for all } m_j.
			\end{equation}
			Let us assume that $z_1\geq0$.  
			\emph{Sub-Case 1}:
			The entries of $P$ must be of the following form 
			\begin{equation}\label{ta}
			z_{\lambda_1}<\dots<z_{\lambda_{k-i-1}}<0<z_{\lambda_{k-i}}<\dots< z_{\lambda_{k-i+j}}<-E_{l}/2< z_{\lambda_{k-i+j+1}}<\dots <z_{\lambda_{k-1}}=z_1<\dots<z_{\lambda_{n}},
			\end{equation}
			where $1\leq i\leq k-1$ and $0\leq j\leq i-2$.
			
			Let \[m:=z_{\lambda_{k-i+j+1}}+E_l/2.\]
			Choose $\epsilon>0$, small enough so that
			\[\dots<z_{\lambda_{k-i+j}}<z_{\lambda_{k-i+j+1}}-m-\epsilon<-E_l(P)/2<z_{\lambda_{k-i+j+1}}<\dots\]
			For $l\notin\{\lambda_{k-i+j+1},\dots,\lambda_{k-1}\}$,
			pick $$z_{\lambda_{k-i+j+1}}'=z_{\lambda_{k-i+j+1}}-m-{\epsilon}=\frac{-E_l}{2}-\epsilon,$$
			$$z_{\lambda_{k-i+j+2}}'=z_{\lambda_{k-i+j+1}}-m-{\epsilon}/2=\frac{-E_l}{2}-\epsilon/2,$$
			and finally
			$$z_{\lambda_{k-2}}'=z_{\lambda_{k-i+j+1}}-m-\frac{\epsilon}{j-i}=\frac{-E_l}{2}-\frac{\epsilon}{j-i}$$
			and $z_{\lambda_t}'=z_{\lambda_t}$, for all  $t\leq k-i+j$ and  $n>t\geq\lambda_{k}$ and 
			$$z_n'=z_n+\sum_{2\leq t\leq i-j-1}(z'_{\lambda_{k-t}}-z_{\lambda_{k-t-1}}).$$ 
			In particular, one has 
			\begin{equation}\label{tataa}
			z_t'      \leq z_{t}                        \end{equation}
			Since $E_{\lambda_k}>E_{\lambda_{k+1}}$ for all $k$, one necessarily has 
			\begin{equation}
			l\neq \lambda_n \hspace{3pt} \text{and }  m_j\neq \lambda_n \text{ for all j}. 
			\end{equation}
			Hence for $l\notin\{\lambda_{k-i+j+1},\dots,\lambda_{k-1},\lambda_n\}$,
			\begin{align*}
			\tilde{B}_{l}^{n+1}(Q)=&2z_1'+2z_l'+\sum_{t\neq1,l}z_t'\\
			=&2z_1'+2z_l'+\sum_{t\in[1, k-i+j]\cup[k+1,n-1]}z_{\lambda_t}+ \left(z_{\lambda_{k-i+j+1}}'+\dots+ z_{\lambda_{k-1}}'\right)+z_{\lambda_n}'\\
			=& 2z_1'+E_l(Q)\\
			=&2z_1'+E_l(Q)=0
			\end{align*} 	
			On the other hand if there exist a  $j_0$ such that $m_{j_0}\in\{\lambda_{k-i+j+1},\dots,\lambda_{k-1}\}$,
			\begin{align*}
			\tilde{B}_{\lambda_{m_{j_0}}}^{n+1}(Q)=& 2z_1'+E_{\lambda_{m_{j_0}}}(Q)\\
			=&2z_1'+2z_{\lambda_{m_{j_0}}}'+\sum_{t\neq \lambda_{m_{j_0}}, 1} z_t'\\
			=&2z_1'+2z_{\lambda_{m_{j_0}}}+\sum_{t\neq \lambda_{m_{j_0}}, 1}z_t-(z_{\lambda_{m_{j_0}}}-z_{\lambda_{m_{j_0}}}'
			)\\
			=&2z_1'+E_{\lambda_{m_{j_0}}}(P)-(z_{\lambda_{m_{j_0}}}-z_{\lambda_{m_{j_0}}}')\\
			=&-E_l(P)+E_{m_{j_0}}(P)-(z_{\lambda_{m_{j_0}}}-z_{\lambda_{m_{j_0}}}')<0 \hspace{3pt} \text{ by \eqref{taa1}} \text{ and }\eqref{tataa}.
			\end{align*}
			For $m_{j_0}<\lambda_{k-i+j}$ and $\lambda_{k-1}<m_{j_0}<\lambda_n$,
			\begin{align*}
			\tilde{B}_{m_{j_0}}^{n+1}(Q)=& 2z_1'+E_{\lambda_{m_{j_0}}}(Q)\\
			=&2z_1'+2z_{m_{j_0}}'+\sum_{t\neq \lambda_{m_{j_0}}, 1}z_t'\\
			=&2z_1'+2z_{m_{j_0}}+\sum_{t\neq \lambda_{m_{j_0}}, 1}z_t\\
			=&2z_1'+E_{\lambda_{m_{j_0}}}(P)\\
			=&-E_l(P)+E_{m_j}(P) <0 \hspace{3pt} (by  \eqref{taa1}).
			\end{align*}
			If $l\in\{\lambda_{k-i+j+1},\dots,\lambda_{k-1}\}$, in particular $l=\lambda_{k-i+j+t_0}$ where $1\leq t_0\leq i-j-1$. Then define $z'_{\lambda_t}=z_{\lambda_t}$ as in the previous case for $t\neq\lambda_n$ and  $$z_n':=z_n+\sum_{t\neq 1,j-i+t_0}\left(z_{\lambda_{k-t}}-z_{\lambda_{k-t-1}}\right)+2\left(z_{\lambda_{k-i+j+t_0}}-z_{\lambda_{k-i+j+t_0-1}}\right).$$
			Now one necessarily has $\tilde{B}_{l}^{n+1}(Q)=0$. Since $B_{\lambda_{i}}^n(P)<B_{\lambda_{i+1}}^n(P)$ for all $i$, one necessarily has $B_{k-i+j+t_0-1}^{n}(Q)<0$. Hence it enough to show that  $$\tilde{B}_{\lambda_{k-i+j+t_0-1}}^{n+1}(Q)<0.$$
			Now
			\begin{align*}
			\tilde{B}_{\lambda_{k-i+j+t_0-1}}^{n+1}(Q)=& 2z_1'+E_{\lambda_{\lambda_{k-i+j+t_0-1}}}(Q)\\
			=&2z_1'+2z_{\lambda_{\lambda_{k-i+j+t_0-1}}}'+\sum_{t\neq \lambda_{\lambda_{k-i+j+t_0-1}}, 1}z_t'\\
			=&2z_1'+2z_{\lambda_{\lambda_{k-i+j+t_0-1}}}+\sum_{t\neq \lambda_{k-i+j+t_0-1}, 1}z_t\\&-\left(z_{\lambda_{\lambda_{k-i+j+t_0-1}}}-z_{\lambda_{\lambda_{k-i+j+t_0-1}}}'\right)+\left(z_{\lambda_{\lambda_{k-i+j+t_0}}}-z_{\lambda_{\lambda_{k-i+j+t_0}}}'\right)\\
			=&2z_1'+E_{\lambda_{k-i+j+t_0-1}}(P)+\left(z_{\lambda_{\lambda_{k-i+j+t_0}}}-z_{\lambda_{\lambda_{k-i+j+t_0-1}}}\right)\\&+\left(z_{\lambda_{\lambda_{k-i+j+t_0-1}}}'-z_{\lambda_{\lambda_{k-i+j+t_0}}}'\right)\\
			=&2z_1'+E_{\lambda_{k-i+j+t_0}}(P)+\left(z_{\lambda_{\lambda_{k-i+j+t_0-1}}}'-z_{\lambda_{\lambda_{k-i+j+t_0}}}'\right) <0. \hspace{3pt} 
			\end{align*}
			\item Case [II] 
			
			Suppose the entries of the point $P\in C$ has the following representation
			\begin{equation}\label{ra14}
			z_{\lambda_1}< z_{\lambda_2}= \dots= z_{\lambda_{i_1}}< z_{\lambda_{i_1+1}}=z_{\lambda_{i_2}}<\dots< z_{\lambda_{n}}.
			\end{equation}
			The entries of the sign vector $v(C)$ corresponding to $B_i^n$ has the following possibilities. Either 
			\begin{enumerate}
				\item all of the  $v_{B_i^{n}}(P)>0$,
				\item all of  the  $v_{B_i^{n}}(P)<0$, or
				\item few of	$v_{B_{i_j}^{n}}(P)>0$ and
				$v_{B_{m_j}^n}(P)<0$,
				\item 	$v_{B_{i_j}^{n}}(P)>0$, $v_{B_{l_j}^{n}}(P)=0$ and
				$v_{B_{m_j}^n}(P)<0$, where $1\leq i_j$, $l$, $m_j\leq n$ for all $j$. 
			\end{enumerate}	
			Our claim is the following 
			\begin{enumerate}
				\item[I.] For all $1\leq i\leq n$, if $v_{B_i^{n}}(P)>0$,  then $v_{\tilde{B}_i^{n+1}}(Q)>0$,
				\item[II.]  For all $1\leq i\leq n$, if $v_{B_i^{n}}(P)<0$,  then $v_{\tilde{B}_i^{n+1}}(Q)<0$ or
				\item[III.]  If	$v_{B_{i_j}^{n}}(P)>0$ and
				$v_{B_{m_j}^n}(P)<0$, then $v_{\tilde{B}_{i_j}^{n+1}}(Q)>0$ and
				$v_{\tilde{B}_{m_j}^{n+1}}(Q)<0$  where $1\leq i_j$, $m_j\leq n$. 
				\item[IV.] If $v_{B_{i_j}^{n}}(P)>0$, $v_{B_{l_j}^{n}}(P)=0$ and
				$v_{B_{m_j}^{n}}(P)<0$, then $v_{\tilde{B}_{i_j}^{n+1}}(Q)>0$, $v_{\tilde{B}_{l_j}^{n+1}}(Q)=0$ and
				$v_{\tilde{B}_{m_j}^{n+1}}(Q)<0$,
			\end{enumerate}
			The proof of claim [I], [II]and [III] is similar to the case \ref{ita1}. The proof of claim [IV] can be reduced to the proof in Case [I] as follows.
			One has 
			$$v_{B_{i_1}^n}(P)=v_{B_{i_1}^n}(P)=0$$ 
			implies $z_{i_1}=z_{i_2}$. Moreover $v_{B_{l_1}^n}(P)<v_{B_{l_2}^n}(P)$ if  $z_{l_1}<z_{l_2}$. Therefore identifying equalities in \eqref{ra14} can be reduced to \eqref{ra11} and the condition (3) can be reduced to the Case[I]. 
		\end{itemize}
	\end{proof}
	Recall that the hyperplane arrangement  $\mathcal{H}_{\R^n}$ and $\mathcal{A}_n$ are isomorphic since one is the pullback of the other under the natural linear  isomorphism $\phi_n$, for all $n$. For any $1\leq i<j\leq n+1$, let  $W_{n}^{i,j}:=\R^{n+1}\cap \phi_{n+1}^*(L_{i,j}^{n+1})$ be a linear subspace of dimension $n$ with an  arrangement $\mathcal{H}_{W_{n}^{i,j}}$ induced by $\mathcal{A}_{n+1}$.  Similar to the case $i=1$ and $j=n+1$ there is a natural isomorphism of linear spaces $\R^{n}\longrightarrow W_{n}^{i,j}$ which induces an embedding of corresponding posets.
	\begin{thm}\label{ita}
		Let $W_{n}^{i,j}:=\R^{n+1}\cap A_{i,j}^{n+1}$. There is an order preserving bijection $\iota_n^{i,j}:\mathcal{C}_{\R^{n}}\longrightarrow \mathcal{C}_{W_{n}^{i,j}}$ which induces an order preserving embedding 
		\[\iota_{L_{i,j}^{n+1}}:\mathcal{C}_{\mathcal{A}_{n}}\longrightarrow \mathcal{C}_{\mathcal{A}_{n+1}}.\]	
	\end{thm}
	\subsection{Double representation for Coxeter arrangement}
	
	From Theorem \ref{tthm}, the abelian category $\mathscr{P}er(V_{\C}^n,S^n)$ is equivalent to the abelian category  $\mathcal{J}_{\mathcal{A}_n}$, the full subcategory of double representations of $\mathcal{C}_{\mathcal{A}_n}$ satisfying \emph{monotonicity}, \emph{transitivity} and \emph{invertibility}. Let us denote  $\mathcal{J}_{n}:=\mathcal{J}_{\mathcal{A}_n}$. 
	Reacll that the hyperplane arrangement $\mathcal{H}_{\mathcal{A}_n}$ is isomorphic to the hyperplane arrangement $\mathcal{H}_{\R^n}$ in $\R^{n}$ for all $n$ and   $\mathcal{H}_{\mathcal{A}_{n+1}}$ induces an arrangement on $L^{n+1}_{i,j}$. The induced arrangement on $L^{n+1}_{i,j}$ is isomorphic to $\mathcal{H}_{W_n^{i,j}}$. By Theorem $\ref{ita}$  there is an order preserving bijection between posets of  faces of the above two arrangements $\mathcal{H}_{\R^n}$ and $\mathcal{H}_{W_n^{i,j}}$. In particular, there is an order preserving embedding of posets $$\iota_{L_{i,j}^{n+1}}:\mathcal{C}_{\mathcal{A}_{n}}\hookrightarrow \mathcal{C}_{\mathcal{A}_{n+1}}.$$          
	Define $$\tilde{\Phi}_{L_{i,j}^{n+1}}:\Rep_2(\mathcal{A}_{n})\rightarrow \Rep_2(\mathcal{A}_{n+1})$$
	$$\psi=(E_C,\gamma_{C'C},\delta_{CC'})_{(C,C'\in\mathcal{C}_{\mathcal{A}_{n}})}\mapsto \tilde{\Phi}_{L_{i,j}^{n+1}}\psi:=(\tilde{E}_D,\tilde{\gamma}_{D'D},\tilde{\delta}_{DD'})_{(D,D'\in\mathcal{C}_{\mathcal{A}_{n+1}})},$$
	where \begin{align*}
	\tilde{E}_D&=E_{\iota_{L_{i,j}^{n+1}}(\tilde{C})} \hspace{10pt} \text{ if }  D=\iota_{L_{i,j}^{n+1}}(\tilde{C})\\
	&=0, \text{          otherwise.}  
	\end{align*}

	\begin{prop}\label{thm}
		For any $1\leq i<j\leq n+1$,  $\tilde{\Phi}_{L_{i,j}^{n+1}}$ induces an exact and fully-faithful functor 
		$\Phi_{L_{i,j}^{n+1}}:\mathcal{J}_n\rightarrow \mathcal{J}_{n+1}$.
	\end{prop}
	\begin{proof}
		
		Let $\mathcal{J}_{n+1|L^{n+1}_{i,j}}$ be the full subcategory of $\mathcal{J}_{n+1}$ consisting of double representations of poset associated to the induce hyperplane arrangement on $L^{n+1}_{i,j}$.  
		Thus any element of $\mathcal{J}_{n+1|L^{n+1}_{i,j}}$ can be identified with  $(E_C,\gamma_{C'C},\delta_{CC'})_{(C,C'\in \bar{\mathcal{C}}_{\mathcal{A}_{n}} )}$ where $ \bar{\mathcal{C}}_{\mathcal{A}_{n}} := \iota_{L_{i,j}^{n+1}}(\mathcal{C}_{\mathcal{A}_{n}}) $. 
		In order to prove $\tilde{\Phi}_{L_{i,j}^{n+1}}$ induces the following functor
		\[\Phi_{L_{i,j}^{n+1}}:\mathcal{J}_{n}\rightarrow \mathcal{J}_{n+1|L^{n+1}_{i,j}} \subseteq \mathcal{J}_{n+1},\] it is enough to show that  $\tilde{\Phi}_{L_{i,j}^{n+1}}(\psi)$ satisfies monotonicity, transitivity and invertibility. For simplicity, we consider the case $i=1$ and $j=n+1$.
		
		\emph{Monotonicity} $:$ For any   $\psi=(E_C,\gamma_{C'C},\delta_{CC'})_{(C,C'\in\mathcal{C}_{\mathcal{A}_{n}})}\in\mathcal{J}_n$, by montonicity one has $$\gamma_{C'C}\delta_{CC'}=Id_{E_C},$$ where $C'$, $C\in \mathcal{C}_{\mathcal{A}_{n}}$  satisfying $C'\leq C$.  
		Hence $\iota_{L_{1,n+1}^{n+1}}(C')\leq \iota_{L_{1,n+1}^{n+1}}(C)$. Thus $\tilde{\Phi}_{L_{i,j}^{n}}(\psi)$ satisfies  monotonicity relation for $D=\iota_{L_{1,n+1}^{n+1}}(C)$ and $D'=\iota_{L_{1,n+1}^{n+1}}(C')$.  For any $D\subset (V^{n+1}\setminus L^{n+1}_{1,n+1} )$, one has  $E_D=0$ and $\delta_{DD'}$ is the zero map.  Hence  monotonicity is satisfied by $\Phi_{L_{1,n+1}^{n+1}}(\psi)$.


		\emph{Transitivity} $:$ For a collinear triple $(A,B,C)\in \mathcal{C}_{\mathcal{A}_n}^3$, using the transitivity relation on $\phi$ one obtains  $\phi_{AC}=\phi_{BC}\phi_{AB}$. In order to check transitivity for $\Phi_{L_{1,n+1}^{n+1}}(\psi)$, it is enough to prove transitivity for all collinear triples in $\mathcal{C}_{\mathcal{A}_{n+1}}^3$. 
		
		If $A$, $B$ and $C \subset L^{n+1}_{1,n+1}$, then collinear triple $(A,B,C)$ is a image of a collinear triple in $\mathcal{C}_{\mathcal{A}_{n}}$. Hence transitivity of such a triple is satisfied by the transitivity relation of $\psi$.  If any two of $A$, $B$ and $C$ lies in $L^{n+1}_{1,n+1}$ then the third one must lie on $L^{n+1}_{1,n+1}$, as $L^{n+1}_{1,n+1}$ is a linear subspace of $V^{n+1}$. Hence the transitivity  is satisfied for such triple follows from the previous case. If $A$, $C\subset (V^{n+1}\setminus L^{n+1}_{1,n+1})$ and $B\subset L^{n+1}_{1,n+1}$, then from definition it follows that $E_A=E_C=0$ with  $\phi_{AC}=0$ and $\phi_{AB}=0$. Hence $\phi_{AC}=\phi_{BC}\phi_{AB}=0$. Similarly  if none of the $A$, $B$ and $C$ lies in $L^{n+1}_{1,n+1}$, then $E_A=E_B=E_C=0$ and  $\phi_{AC}=\phi_{BC}=\phi_{AB}=0$. Thus   transitivity holds for $\Phi_{L_{1,n+1}^{n+1}}(\psi)$.

		\emph{Invertibility} $:$ Let $C_1$ and $C_2$ be two faces opposes through  a $(d-1)$-dimensional face $B$. Then $\phi_{C_1C_2}$ is an isomorphism. Thus either $C_1$ and $C_2$ both lies in $L^{n+1}_{1,n+1}$ or both do not lie in $L^{n+1}_{1,n+1}$. In any case it is an isomorphism.

	\end{proof}

	\begin{thm}
		For any $1\leq i<j\leq n+1$,	there  is a fully-faithful exact functor compatible with the Verdier duality
		
		\begin{equation}
		\Lambda_{L_{i,j}^{n+1}}:\mathscr{P}er(V_{\C}^n,\mathcal{S}^n)\rightarrow \mathscr{P}er(V_{\C}^{n+1},\mathcal{S}^{n+1}),
		\end{equation}
		which makes the following diagram commutative
		$$
		\begin{tikzcd}
		\mathscr{P}er(V_{\C}^n,\mathcal{S}^n)\arrow{r}{\mathcal{Q}^n}\arrow{d}{\Lambda_{L_{i,j}^{n+1}}}& \mathcal{J}_i\arrow{d}{\Phi_{L_{i,j}^{n}}}\\
		\mathscr{P}er(V_{\C}^{n+1},\mathcal{S}^{n+1})\arrow{r}{\mathcal{Q}^{n+1}}& \mathcal{J}_{n+1}.
		\end{tikzcd}
		$$ 
		Moreover image  of simple objects are simple.
	\end{thm}
	\begin{proof}
We have the following commutative diagram
$$
\begin{tikzcd}
\mathscr{P}er(V_{\C}^n,\mathcal{S}^n)\arrow{r}{\mathcal{Q}^n}\arrow{d}{\cong}& \mathcal{J}_n\arrow{d}{\Phi_{L_{i,j}^{n}}}\\
\mathscr{P}er(L^{n+1}_{i,j},\mathcal{S}_{|L^{n+1}_{i,j}}^{n+1})\arrow{r}{\mathcal{Q}^{n+1}_{|L^{n+1}_{i,j}}}\arrow{d}{(h_{(-1)}\circ k_*)[-1]}& \mathcal{J}_{n+1|L^{n+1}_{i,j}}\arrow{d}{r_*}\\
\mathscr{P}er(V_{\C}^{n+1},\mathcal{S}^{n+1})\arrow{r}{\mathcal{Q}^{n+1}}& \mathcal{J}_{n+1}.
\end{tikzcd}
$$
In this diagram, $r_*$ is the right adjoint of the hyperbolic restriction, [Sec 5.A, \cite{KV}], where $k:L_{i,j \mathbb C} \hookrightarrow \mathbb R^n + i L_{i,j}$ and  $h :R^n + i L_{i,j}\hookrightarrow V_\mathbb C^n$. Let $h_{(-1)}$ be the right adjoint of $h^!$.

The functor  $\mathcal{Q}^{n+1}_{|L^{n+1}_{i,j}}$ and $\Phi_{L_{i,j}^{n}}$ are equivalence of categories (by \ref{tthm} and \ref{thm}). Let $\Theta^{n+1}$ be the inverse of the functor $\mathcal{Q}^{n+1}$ and  $\Theta^{n+1}_{|L^{n+1}_{i,j}}$ be the the inverse of $\mathcal{Q}^{n+1}_{|L^{n+1}_{i,j}}$, the restriction   $\mathcal{Q}^{n+1}$.
Thus the top left vertical functor	
$$\Theta^{n+1}_{|L^{n+1}_{i,j}} \circ  \Phi_{L_{i,j}^{n}}\circ\mathcal{Q}^n$$ is an equivalence of categories.  

We can check that the lower square commutes with hyperbolic extensions $r^*$ on right and $(k^*\circ h^{!}[1])$ on left, [Prop 5.1, \cite{KV}]. Hence using the adjunction required square commutes i.e. $r_* \circ \mathcal{Q}^{n+1}_{|L^{n+1}_{i,j}} = \mathcal{Q}^{n+1} \circ ((h_{(-1)}\circ k_*) [-1])$.
Therefore we can define	 
$$\Lambda_{L_{i,j}^{n+1}}:=\Theta^{n+1}\circ r_*\circ  \Phi_{L_{i,j}^{n}}\circ\mathcal{Q}^n,$$  
which gives a fully faithful exact functor with the required commutative square. The compatibility with the Verdier duality follows from the fact that all three arrows are compatible with duality. The last assertion follows from the fact that any sub-object of the image is again in the image.
	\end{proof}	
	
	\begin{cor}
		Let $\mathcal{F}$ be a constructible sheaves on the hyperplane arrangement $\mathcal{A}_{n+1}$. Then we have the following.
		\begin{enumerate}
			\item  If $\mathcal{F}\in \mathscr{P}er(V_{\C}^{n+1},\mathcal{S}^{n+1})$ is irreducible then  $\dim_{\C}(E_C(\mathcal{F}))=0$ or $1$, where $C$ be any open cell in $\mathcal{C}_{\mathcal{A}_{n+1}}$.   
			\item  Suppose $\mathcal{F}\in \mathscr{P}er(V_{\C}^{n+1},\mathcal{S}^{n+1})$ is irreducible and  $\dim_{\C}(E_C(\mathcal{F}))=0$ where $C$ be any open cell in  $\mathcal{C}_{\mathcal{A}_{n+1}}$. Then there exists $\mathcal{G}\in \mathscr{P}er(V_{\C}^{n},\mathcal{S}^{n})$ and $L_{i,j}^{n+1}\in \mathcal{H}_{\mathcal{A}_{n+1}}$ such that  \[\Lambda_{L_{i,j}^{n+1}}(\mathcal{G})=\mathcal{F}.\]
		\end{enumerate}
	\end{cor}
	\begin{proof}
		\begin{enumerate}
			\item For any two open cell $C$ and $C'\in \mathcal{C}_{\mathcal{A}_n}$, there exist $$\phi_{CC'}:E_{C}(\mathcal{F})\longrightarrow E_{C'}(\mathcal{F}),$$ which is an isomorphism by invertibility. Thus if $E_{\lambda}\subset E_{C}(\mathcal{F}) $ be a proper subspace, one necessarily has a proper subspace $\phi_{CC'}(E_{\lambda})\subset E_{C'}(\mathcal{F})$. Since,   
			by monotonicity, for any faces $D$, $C$ satisfying $D\leq C$, one has a injective map $\delta_{DC}:E_{C}(\mathcal{F})\longrightarrow E_{D}(\mathcal{F})$. Hence one can generate a proper subrepresentation of $\left(E_{C}(\mathcal{F}),\delta_{C'C},\gamma_{CC'}\right)_{C\in \mathcal{C}_{\mathcal{A}_n} }$ by taking translates of $E_{\lambda}$ under various maps $\delta_{\star,\star}$.

			Thus if $\left(E_{C}(\mathcal{F}),\delta_{C'C},\gamma_{CC'}\right)_{C\in \mathcal{C}_{\mathcal{A}_n} }$  is irreducible then  $E_C(\mathcal{F})$ is irreducible  if and only if $\dim_{\C}(E_C(\mathcal{F}))=0$ or $1$.   
			
			\item Suppose $\mathcal{F}$ is irreducible and there exist two distinct codimension one cell $D_1$ and $D_2$ and two distinct hyperplanes $H_1$ and $H_2\in \mathcal{A}_{n+1}$ such that for all $i$
			\begin{align*}
			E_{D_i}(\mathcal{F})&\neq 0\\
			D_i\subset& H_i
			\end{align*}
			There exist an open cell $C$ in  $\mathcal{A}_{n+1}$ such that  $(D_1,C,D_2)$ form a collinear triple. Indeed, any line joining $d_1\in D_1$ and $d_2\in D_2$ must intersect an open cell, Since $D_i$ is a connected component of $H_i\setminus\cup_{H_j(\neq H_i)\in \mathcal{A}_{n+1}}\left(H_j\cap H_k \right)$ for $i\in\{1,2\}$.  Then by transitivity 
			\[\phi_{D_1D_2}=\phi_{C_2D_2}\phi_{D_1C_2}=0\]
			
			Similarly there exist $L\subset H_1\cap H_2$, a cell of codimension one in  $H_1$ and $H_2$ such that $(D_1,L,D_2)$ forms a collinear triple. Then $\phi_{D_1D_2}=\gamma_{LD_2}\delta_{D_1L}=0$.
			Hence $$\Im(\delta_{D_1L})\subset \Ker(\gamma_{LD_2}).$$
			In other words $E_L(\mathcal{F})$ contains $\Im(\delta_{D_1L})$ and $\Im(\delta_{D_2L})$ as direct summand. Similarly it happen in all other cells contained in $H_1\cap H_2$. Hence $E_C(\mathcal{F})$ can not be irreducible, which a contradiction. 
		\end{enumerate}    
	\end{proof}

\begin{cor}
	Let $\mathcal{F}\in \mathscr{P}er(V_{\C}^{2},\mathcal{S}^{2})$ be be a perverse sheaves on the hyperplane arrangement $\mathcal{A}_{2}$. If $\mathcal{F}$ is irreducible and $\dim E_{C}(\mathcal{F})=1$, where $C$ be an open cells, then $\dim E_{D}(\mathcal{F})\leq 2$,  for any cell $D$ of dimension one. 
\end{cor}
\begin{proof}
	
Let $C_1$, $C_2$ be two open cell and $D$ be a cell of dimension one in $\mathcal{A}_2$ satisfying 
$$C_1\geq D\leq C_2.$$
Hence by monotonicity, $$E_{C_1}(\mathcal{F})\hookrightarrow E_{D}(\mathcal{F}) \text{ and }E_{C_2}(\mathcal{F})\hookrightarrow E_{D}(\mathcal{F}).$$
Since $C_1$ and $C_2$ are opposed by $D_1$, by invertibility, $$\phi_{DC_2}\circ\phi_{C_1D}:E_{C_1}(\mathcal{F})\cong E_{C_2}(\mathcal{F})$$
is an isomorphism. Moreover $\phi_{C_1O}:E_{C_1}(\mathcal{F})\rightarrow  E_{O}(\mathcal{F})$ can be factorised by $\phi_{C_1O}=\phi_{DO}\phi_{C_1D}$. Similarly $\phi_{C_2O}=\phi_{DO}\phi_{C_2D}$. 
Moreover, for any open cells $C$ and $C'$, one necessarily has $E_{C}(\mathcal{F})\cong E_{C'}(\mathcal{F})$. Observe that any cell $D$ of dimension one can be dominated by exactly two open cells. 
Therefore	image of $\phi_{C_1 D}$ and image of $\phi_{C_2 D}$ will generate a subspace of  $E_{D}(\mathcal{F})$ of dimension at most two. Since $\mathcal{F}$ is irreducible, $\mathcal{Q}(\mathcal{F})$ is irreducible (Theorem \ref{tthm}). Hence image of $\phi_{C_2 D}$ and $\phi_{C_2 D}$ would generate $E_{D}(\mathcal{F})$.

\end{proof}

\begin{rem}
	Since the classification of irreducible perverse sheaves of Coxeter arrangement of type $\mathcal{A}_2$ was given by Rowley under the assumption that $\dim(E_D(\mathcal{F})\leq2$ for any cell $D$ of dimension one in $\mathcal{A}_2$. Using the above corollary we obtain that this classification is a complete classification of irreducible perverse sheaves in $\mathcal{A}_2$.
\end{rem}

\subsection*{Acknowledgements} 

Both the authors are thankful to Harish-Chandra Research Institute. The second author  is also thankful to KSCSTE-Kerala School of Mathematics for their support.

\bibliographystyle{alpha}


\end{document}